\documentclass[a4paper,12pt]{amsart}

\usepackage{float}
\usepackage{euscript,eufrak,verbatim}
\usepackage{graphicx}
\usepackage[usenames]{color}
\usepackage[colorlinks,linkcolor=red,anchorcolor=blue,citecolor=blue]{hyperref}
\usepackage{amsmath}
\usepackage{amsthm}
\usepackage[all]{xy}
\usepackage{graphicx} 
\usepackage{amssymb} 
\usepackage{amsfonts}
\usepackage{enumitem}
\usepackage{mathrsfs}
\usepackage{amscd}
\usepackage{pifont}

\usepackage[pagewise]{lineno}

\makeatletter
\newcommand{\svdots}{%
  \vbox{\fontsize{\sf@size}{\sf@size pt}\linespread{0.3}\selectfont
    \kern0.2\baselineskip
    \hbox{.}\hbox{.}\hbox{.}%
    \kern0.1\baselineskip
  }%
}
\makeatother

\newtheorem*{conj}{Conjecture}
\newtheorem{theorem}{Theorem}

\theoremstyle{plain}
\newtheorem{thm}{\protect\theoremname}[section]
\theoremstyle{plain}
\newtheorem*{thm*}{\protect\theoremname}
\theoremstyle{plain}
\newtheorem{lem}[thm]{\protect\lemmaname}
\theoremstyle{definition}

\theoremstyle{definition}

\theoremstyle{plain}
\newtheorem{prop}[thm]{\protect\propositionname}
\theoremstyle{definition}
\newtheorem*{rem*}{\protect\remarkname}
\theoremstyle{plain}

\theoremstyle{plain}
\newtheorem*{fact*}{\protect\factname}
\theoremstyle{definition}

\usepackage{babel}
\providecommand{\corollaryname}{Corollary}
  \providecommand{\definitionname}{Definition}
  \providecommand{\factname}{Fact}
  \providecommand{\lemmaname}{Lemma}
  \providecommand{\propositionname}{Proposition}
  \providecommand{\remarkname}{Remark}
\providecommand{\theoremname}{Theorem}

\makeatletter
\def\blfootnote{\xdef\@thefnmark{}\@footnotetext}
\makeatother

\newcommand{\Z}{\mathbb{Z}}
\newcommand{\N}{\mathbb{N}}
\newcommand{\PP}{\mathbb{P}}

\usepackage{babel}
\providecommand{\corollaryname}{Corollary}
  \providecommand{\definitionname}{Definition}
  \providecommand{\factname}{Fact}
  \providecommand{\lemmaname}{Lemma}
  \providecommand{\propositionname}{Proposition}
  \providecommand{\remarkname}{Remark}
\providecommand{\theoremname}{Theorem}

\usepackage{babel}
\providecommand{\corollaryname}{Corollary}
  \providecommand{\definitionname}{Definition}
  \providecommand{\factname}{Fact}
  \providecommand{\lemmaname}{Lemma}
  \providecommand{\propositionname}{Proposition}
  \providecommand{\remarkname}{Remark}
\providecommand{\theoremname}{Theorem}

\usepackage{babel}
\providecommand{\corollaryname}{Corollary}
  \providecommand{\definitionname}{Definition}
  \providecommand{\factname}{Fact}
  \providecommand{\lemmaname}{Lemma}
  \providecommand{\propositionname}{Proposition}
  \providecommand{\remarkname}{Remark}
\providecommand{\theoremname}{Theorem}

\usepackage{babel}
\providecommand{\corollaryname}{Corollary}
  \providecommand{\definitionname}{Definition}
  \providecommand{\factname}{Fact}
  \providecommand{\lemmaname}{Lemma}
  \providecommand{\propositionname}{Proposition}
  \providecommand{\remarkname}{Remark}
\providecommand{\theoremname}{Theorem}

\usepackage{babel}
\providecommand{\corollaryname}{Corollary}
  \providecommand{\definitionname}{Definition}
  \providecommand{\factname}{Fact}
  \providecommand{\lemmaname}{Lemma}
  \providecommand{\propositionname}{Proposition}
  \providecommand{\remarkname}{Remark}
\providecommand{\theoremname}{Theorem}

\usepackage{babel}
\providecommand{\corollaryname}{Corollary}
  \providecommand{\definitionname}{Definition}
  \providecommand{\factname}{Fact}
  \providecommand{\lemmaname}{Lemma}
  \providecommand{\propositionname}{Proposition}
  \providecommand{\remarkname}{Remark}
\providecommand{\theoremname}{Theorem}

\usepackage{babel}
\providecommand{\corollaryname}{Corollary}
  \providecommand{\definitionname}{Definition}
  \providecommand{\factname}{Fact}
  \providecommand{\lemmaname}{Lemma}
  \providecommand{\propositionname}{Proposition}
  \providecommand{\remarkname}{Remark}
\providecommand{\theoremname}{Theorem}

\usepackage{babel}
\providecommand{\corollaryname}{Corollary}
  \providecommand{\definitionname}{Definition}
  \providecommand{\factname}{Fact}
  \providecommand{\lemmaname}{Lemma}
  \providecommand{\propositionname}{Proposition}
  \providecommand{\remarkname}{Remark}
\providecommand{\theoremname}{Theorem}

\usepackage{babel}
\providecommand{\corollaryname}{Corollary}
  \providecommand{\definitionname}{Definition}
  \providecommand{\factname}{Fact}
  \providecommand{\lemmaname}{Lemma}
  \providecommand{\propositionname}{Proposition}
  \providecommand{\remarkname}{Remark}
\providecommand{\theoremname}{Theorem}

\usepackage{babel}
\providecommand{\corollaryname}{Corollary}
  \providecommand{\definitionname}{Definition}
  \providecommand{\factname}{Fact}
  \providecommand{\lemmaname}{Lemma}
  \providecommand{\propositionname}{Proposition}
  \providecommand{\remarkname}{Remark}
\providecommand{\theoremname}{Theorem}

\usepackage{babel}
\providecommand{\corollaryname}{Corollary}
  \providecommand{\definitionname}{Definition}
  \providecommand{\factname}{Fact}
  \providecommand{\lemmaname}{Lemma}
  \providecommand{\propositionname}{Proposition}
  \providecommand{\remarkname}{Remark}
\providecommand{\theoremname}{Theorem}

\usepackage{babel}
\providecommand{\corollaryname}{Corollary}
  \providecommand{\definitionname}{Definition}
  \providecommand{\factname}{Fact}
  \providecommand{\lemmaname}{Lemma}
  \providecommand{\propositionname}{Proposition}
  \providecommand{\remarkname}{Remark}
\providecommand{\theoremname}{Theorem}

\makeatother

\providecommand{\corollaryname}{Corollary}
\providecommand{\definitionname}{Definition}
\providecommand{\factname}{Fact}
\providecommand{\lemmaname}{Lemma}
\providecommand{\propositionname}{Proposition}
\providecommand{\remarkname}{Remark}
\providecommand{\theoremname}{Theorem}

\author{
Qiang Huo, Zhengxing Lian and Kangbo Ouyang
}

\DeclareMathOperator{\Prob}{Prob}

\def\R{\mathbb{R}}
\def\E{\mathbb{E}}

\def\N{\mathbb{N}}
\def\Z{\mathbb{Z}}

\def\cP{\mathcal{P}}
\def\D{\mathbf{D}}
\def\H{\mathbf{H}}
\def\I{\mathbf{I}}

\newcommand{\cG}{\mathcal{G}}

\def\00{\mathbf 0}
\title{ Odd order cases of Chowla's averages over arbitrary rectangles}

 \address[Qiang Huo]
 {School of Mathematical Sciences, Beijing University of Posts and Tele-communications, Beijing 100876, China
 \newline \indent Key Laboratory of Mathematics and Information Networks (Beijing University of Posts and Telecommunications), Ministry of Education, Beijing 100876, China}

 \email{qianghuo@ustc.edu.cn}

  \address[Zhengxing Lian]
 {School of Mathematical Sciences, Xiamen University, Xiamen 361005, China}

\email{lianzx@xmu.edu.cn}

 \address[Kangbo Ouyang]
 {School of Mathematical Sciences, University of Science and Technology of China, Hefei 230026, China}

 \email{oy19981231@mail.ustc.edu.cn}
\begin{document}

\maketitle

\begin{abstract}
Let $\lambda$ be the Liouville function. 
In this paper we prove an average form of Chowla Conjecture for even  $k $. that is, for every fixed even integer \(k\ge 2\), we prove that 
\[
\lim_{\min\{M,N\}\to \infty}\frac{1}{MN}\sum_{m=1}^M\sum_{n=1}^N\lambda(m)\lambda(m+a_1n)\cdots\lambda(m+a_kn)=0.
\]
 We construct a stationary model for both averaging variables, derive the finite conditional-entropy estimates. 
 A comparison of overlapping intervals on the logarithmic axis then forces this function to vanish.
\end{abstract}

\section{Introduction}

Let \(\lambda:\mathbb{N}\to\{-1,1\}\) denote the Liouville function, defined by \(\lambda(n)=(-1)^{\Omega(n)}\), where \(\Omega(n)\) is the number of prime factors of \(n\) counted with multiplicity. Equivalently, \(\lambda\) is the completely multiplicative function with \(\lambda(p)=-1\) for every prime \(p\).

The Chowla conjecture asserts that the Liouville function has no nontrivial correlations with its shifts.

\begin{conj}[Chowla]\label{conj:chowla}
Let \(k\ge 1\), and let \(a_1,\dots,a_k\) be distinct nonnegative integers. Then
\[
\sum_{n\le x}\lambda(n+a_1)\cdots\lambda(n+a_k)=o(x).
\]
\end{conj}

One may try to view the Chowla Conjecture through the lens of averaging. In \cite{MRT15}, the authors prove an average form of Chowla's Conjecture:
$$\sum_{1\leq a_1,a_2,\ldots,a_k\leq H}|\sum_{m=1}^M\prod_{j=1}^{k}\lambda(m+a_j)|=o(H^kM).$$
In \cite{F18}, Frantzikinakis asked whether one  can find some increasing sequence \(M_\ell\) such that
\[
\lim_{N\to\infty}\lim_{\ell\to\infty}\frac{1}{M_\ell N}
\sum_{m=1}^{M_\ell}\sum_{n=1}^{N}
\lambda(m)\lambda(m+n)\lambda(m+2n)=0?
\]
By \cite{TT19}, one can give a positive answer to this question. Moreover, in \cite{MRTTZ23}, for any \(\epsilon>0\) and every fixed choice of integers \(0=a_0<a_1<\dots<a_k\),
\[
\lim_{M\to\infty}\frac{1}{[M^\epsilon]}
\sum_{n\le M^\epsilon}\frac{1}{M}\sum_{m=1}^{M}
\prod_{j=0}^{k}\lambda(m+a_j n)=0.
\]
Encouraged by these works, we consider more ``free'' rectangular sums \([1,M]\times[1,N]\). In this paper we prove an average form of the Chowla conjecture for even integer \(k\). More precisely, we establish the following result.

\begin{theorem}\label{main theorem}
    For every fixed even integer \(k\ge 2\) and every fixed choice of integers $0=a_0<a_1<\ldots <a_k$,
\begin{equation}\label{main}
 \lim_{\min\{M,N\}\to \infty}\frac{1}{MN}\sum_{m=1}^M\sum_{n=1}^N\prod_{j=0}^{k}\lambda(m+a_jn)=0.   
\end{equation}
\end{theorem}

Thus, in contrast to the original Chowla conjecture, we allow the shifts to grow linearly with the second variable \(n\), and we average over an arbitrary rectangle \([1,M]\times[1,N]\) without imposing any relation between \(M\) and \(N\). { The result covers all even orders \(k\), including the three-point case \(k=2\), which is the first nontrivial even case.}

Similar to \cite{TT19}, the proof of Theorem \ref{main theorem} relies on Furstenberg-type correspondences and conditional entropy estimates. However, we apply these tools in different settings and with different technical aims. We  build a Furstenberg system with two commuting transformations. For any invariant measure of the system, we prove an information budget and, using Pinsker’s inequality per prime, a difference bound. We then use a continuous scale function $f(t)$ to characterize the limit of the averaged sums, and use overlapping prime intervals to force $f\equiv 0$.


 In this article, we shall use the following notation: For $L\in\N$, denote \([L]=\{1,\dots,L\}\). If \(A\) is a nonempty finite set, write
\[
\E_{a\in A}b(a)=|A|^{-1}\sum_{a\in A}b(a).
\]
If $b$ is a random variable on a finite set $A$, write $\E b =\E_{a\in A}b(a)$. If $P$ is a probability measure on $A$, write 
$$\E_P b=\sum_{a\in A}P(a)b(a).$$
The quantity
\[
D_k(M,N)=\mathbb{E}_{(m,n)\in[M]\times[N]}
\prod_{j=0}^{k}\lambda(m+a_j n)
\]
is the normalized average of the main theorem. The main theorem asserts that this average tends to zero as \(\min(M,N)\to\infty\). For a prime \(p\), we also define
\begin{equation}\label{def of Ekp}
    E_{k,p}(M,N)
=\E_{(m,n)\in [M]\times [N]}
(p^2\mathbf{1}_{p\mid m,p\mid n}-1)
\prod_{j=0}^{k}\lambda(m+a_jn).
\end{equation}
Here the factor \(p^2\mathbf{1}_{p\mid m,\,p\mid n}-1\) compares the original average with the average restricted to the sublattice $(p\Z)\times (p\Z)$. The factor \(p^2\) compensates for the density \(p^{-2}\) of this sublattice, so that \(E_{k,p}(M,N)\) measures the difference between the normalized restricted average and the original average \(D_k(M,N)\).

We now outline the main ideas of the proof. The argument combines the Furstenberg correspondence principle, finite conditional entropy estimates, and a continuous-scale limiting argument. The proof proceeds in several steps.

\medskip
\noindent\textbf{Step 1: Arithmetic reduction.}

Restricting to the sublattice $(p\Z)\times (p\Z)$, the even integer $k$ and using the complete multiplicativity of \(\lambda\), we obtain the exact identity (see Section \ref{Sec: scale function} for details)
\[
E_{k,p}(M,N)=-c_p(M,N)D_k(M_p,N_p)-D_k(M,N),
\]
where \(M_p=\lfloor M/p\rfloor\), \(N_p=\lfloor N/p\rfloor\), and
\[
c_p(M,N)=\frac{p^2M_pN_p}{MN}.
\]
Moreover,
\[
0\le 1-c_p(M,N)\le \frac{p-1}{M}+\frac{p-1}{N}.
\]
This identity shows that \(E_{k,p}\) is controlled by the values of \(D_k\) at the dilated rectangle \((M_p,N_p)\). It will be used to relate the limiting behaviour of \(D_k\) along different scales.

\medskip
\noindent\textbf{Step 2: A dynamical system for the arithmetic variables.}
We encode the Liouville function as a point in the symbol space \(\Sigma=\{-1,1\}^{\mathbb{Z}}\) (We set $\lambda(n)=1$ if $n\leq 0$). Consider the compact metrizable space
\[
X=\Sigma^{k+1}\times K,\qquad
K=\prod_{p\text{ prime}}(\mathbb{Z}/p\mathbb{Z})^2.
\]
On \(X\) we define two commuting transformations \(S\) and \(T\) by
\[
S\bigl((x_j)_{j=0}^{k},(\rho_p,\sigma_p)_p\bigr)
=\bigl((\tau x_j)_{j=0}^{k},(\rho_p+1,\sigma_p)_p\bigr),
\]
\[
T\bigl((x_j)_{j=0}^{k},(\rho_p,\sigma_p)_p\bigr)
=\bigl((\tau^{a_j}x_j)_{j=0}^{k},(\rho_p,\sigma_p+1)_p\bigr),
\]
where \(\tau\) is the shift on \(\Sigma\).  We will construct the empirial measures $\mu_{M,N}^{\lambda}$ such that 
\[
\int F\,d\mu_{M,N}^{\lambda}=D_k(M,N),\qquad
\int G_p\,d\mu_{M,N}^{\lambda}=E_{k,p}(M,N),
\]
where \(F(z)=\prod_{j=0}^{k}x_j(0)\) and \(G_p(z)=(p^2\mathbf{1}_{\rho_p=\sigma_p=0}-1)F(z)\). We will show that any weak-* limit \(\mu\) of a subsequence $\{\mu_{M,N}^{\lambda}\}$ with \(M_\ell,N_\ell\to\infty\) is \(S\)- and \(T\)-invariant, and its projection to \(K\) is the product uniform measure. 

\medskip
\noindent\textbf{Step 3: Finite conditional entropy estimates.}
For a fixed \(S,T\)-invariant measure \(\mu\), we prove a summable bound for the discrepancies \(e_p=\int G_p\,d\mu\) averaged over dyadic intervals of primes. Let
\[
H_r=2^r,\qquad
\mathcal{P}_r=\{p\in\PP:H_r<p\le 2H_r\},\qquad
R_r=|\mathcal{P}_r|.
\]
For 
$
q_r=\frac{1}{R_r}\sum_{p\in\mathcal{P}_r}|e_p|,$
Using information budget and Pinsker's inequality, we prove a bound
\[
\sum_{r=0}^{\infty}\frac{R_r}{H_r}q_r^2
\le 50(k+1)(a_k+1)\log 2.
\]
\medskip
\noindent\textbf{Step 4: A continuous scale function.}
Given any sequence \(M_\ell,N_\ell\to\infty\), we extract a subsequence along which the limit
\[
f(t)=\lim_{\ell\to\infty}
D_k(\lfloor e^{-t}M_\ell\rfloor,\lfloor e^{-t}N_\ell\rfloor)
\]
exists for every \(t\in\mathbb{R}\). The convergence is uniform on compact intervals, and \(f\) is \(4\)-Lipschitz with \(|f(t)|\le 1\). For each fixed \(t\) and each fixed prime \(p\), the arithmetic identity from Step 1 gives
\[
\lim_{\ell\to\infty}E_{k,p}(M_\ell(t),N_\ell(t))
=-f(t+\log p)-f(t).
\]
Applying the entropy estimate from Step 3 to a limiting measure \(\mu_t\) constructed at scale \(t\), we obtain
\[
\sum_{r=0}^{\infty}\frac{R_r}{H_r}
\left(\frac{1}{R_r}\sum_{p\in\mathcal{P}_r}|f(t+\log p)+f(t)|\right)^2
\le 50(k+1)(a_k+1)\log 2.
\]

\medskip
\noindent\textbf{Step 5: Overlapping prime intervals force \(f\equiv 0\).}
At this stage the problem is purely analytic. We have a \(4\)-Lipschitz function \(f:\mathbb{R}\to[-1,1]\) satisfying the above finite-energy condition. We compare prime averages with integrals uniformly for a class of Lipschitz test functions, and then use the prime number theorem to select a common sequence of good dyadic prime intervals for two scales \(t\) and \(t+h\). The key overlap inequality is
\[
(\log 2-h)|f(t+h)-f(t)|\le J_r(t)+J_r(t+h),
\]
where \(J_r(t)\) is an integral version of the prime average. Letting \(r\to\infty\) along the good sequence shows that \(f(t+h)=f(t)\) for every \(t\) and every \(0<h<\log 2\). Hence \(f\) is constant. If \(f\equiv C\), the finite-energy condition becomes
\[
4|C|^2\sum_{r=0}^{\infty}\frac{R_r}{H_r}<\infty.
\]
But by the prime number theorem,
\[
\frac{R_r}{H_r}\sim\frac{1}{r\log 2},
\]
so the sum diverges. Therefore \(C=0\), i.e. \(f\equiv 0\).

\medskip
\noindent\textbf{Step 6: Completion.}
Suppose Theorem~\ref{main theorem} failed. Then there exist \(\delta>0\) and sequences \(M_\ell,N_\ell\ge \ell\) such that
\[
|D_k(M_\ell,N_\ell)|\ge\delta.
\]
Extracting a subsequence as in Step 4, we obtain a continuous scale function \(f\) with \(|f(0)|\ge\delta\). But Step 5 shows that \(f\equiv 0\), a contradiction. This completes the proof of Theorem~\ref{main theorem}.

\medskip

\section{Preliminaries on Dynamical Systems and Shannon Entropy} 

\subsection{A Furstenberg system for arithmetic variables}
In this section, we will construct a topological dynamical system for estimating $D_k(M,N)$ and $E_{k,p}(M,N)$. Indeed, we will give an invariant measure on this topological dynamical system  give an integral expression relating for $D_k(M,N)$ and $E_{k,p}(M,N)$.

Let $\Sigma =\{-1,1\}^{\Z}$, with the shift $(\tau x)(r)=x(r+1)$. Let 
$$K=\prod_{p\in\PP}(\Z/p\Z)^2 \text{ and } X=\Sigma^{k+1}\times K.$$
For any point  $z\in X$, we write  $z $ as $((x_j)_{j=0}^{ k}, (\rho_p,\sigma_p)_{p\in\PP})$. Then $X$ is a compact metrizable space w.r.t. the product topology. 

Now for $z=((x_j)_{j=0}^{ k}, (\rho_p,\sigma_p)_{p\in\PP}) $ we define the actions by
\begin{equation}\label{def of S T}\begin{array}{ll}
    S:X\to X, Sz=((\tau x_j)_{j=0}^{ k}, (\rho_p+1,\sigma_p)_{p\in\PP}), \text{ and }\\
    T:X\to X, Tz=((\tau^{a_j} x_j)_{j=0}^{ k}, (\rho_p,\sigma_p+1)_{p\in\PP}).
    \end{array}
\end{equation}

The inverse of $S,T$ use the inverse symbolic shifts. Note that $S$ and $T$ commute since for each coordinate the actions commute. Their combined action is
\begin{equation}\label{combined action}S^uT^vz=((\tau^{u+a_jv}x_j)_{j=0}^k,(\rho_p+u,\sigma_p+v)_{p\in\PP})\ \ u,v\in \Z.\end{equation}
Now we define continuous functions on $X$ by
\begin{equation}\label{continuous functions}F(z)=\prod_{j=0}^kx_j(0) \text{  and  }G_p(z)=(p^2\mathbf{1}_{\rho_p=\sigma_p=0}-1)F(z).
\end{equation}
For any $w\in \{-1,1\}^{\Z}$, define $\Psi_w(m,n)=((x_j)_{j=0}^{ k}, (\rho_p,\sigma_p)_{p })\in X$ by
\begin{equation}\label{def of Psi}
   x_j(r)=w(m+a_jn+r),\ \rho_p\equiv m(\text{mod}\  p),\ \sigma_p\equiv n(\text{mod}\  p). 
\end{equation}
For all integers $m,n,u,v $, we have that
\begin{equation}\label{Psi equation}
\Psi_w(m+u,n+v)=S^uT^v\Psi_w(m,n). 
\end{equation}
Indeed, the $j $th symbolic coordinate at index \(r\) on the left hand side is
$$
w(m+u+a_j(n+v)+r)=w(m+a_jn+r+u+a_jv),
$$
which equals the coordinate in \eqref{combined action}. The residues on both sides are $m+u$ and $n+v$ modulo $p$.

Now we define probability measure 
\begin{equation}\label{mu a M N}
    \mu_{M,N}^w=\frac{1}{MN}\sum_{m=1}^{M}\sum_{n=1}^{N}\delta_{\Psi_w(m,n)}.
\end{equation}
As $X $ is compact, for any $M_l,N_l\to \infty $, the measures \(\mu_{M_\ell,N_\ell}^w\) have weakly* convergent subsequences. Moreover, we can show that
\begin{lem}\label{invariant measure property}
    If $\mu$ is the limit of weakly* convergent subsequences of some $\mu_{M_l,N_l}^w $, then 
    \begin{itemize}
        \item $\mu$ is invariant under \(S\) and \(T\);
        \item  Let  $\pi:X\to K$ be the projection. Then $\pi_*(\mu)$ the product of the uniform measures on \((\Z/p\Z)^2\). That is, $$\pi_*\mu=\prod_{p\text{ prime }}U_p,$$
        where $U_p $ is the uniform probability measure on the finite set $(\Z/p\Z)^2$.
    \end{itemize}
\end{lem}
\begin{proof}
    By \eqref{Psi equation}, one has that $\Psi_w(m,n)=S^mT^n\Psi(0,0) $. Then by \eqref{mu a M N} 
    $$\mu_{M,N}^w=\frac{1}{MN}\sum_{m=1}^{M}\sum_{n=1}^{N}\delta_{\Psi_w(m,n)}=\frac{1}{MN}\sum_{m=1}^{M}\sum_{n=1}^{N}\delta_{S^mT^n\Psi_w(0,0)}.$$
    As $S$ and $T$ commute and $ \mu$ is the limit of weakly* convergent subsequences of some $\mu_{M_l,N_l}^w $, it is clear that $\mu$ is $S$ and $T$ invariant. 

    Now fix $s\in \N$ and prescribe $$U=\{((x_j)_{j=0}^{ k}, (\rho_p,\sigma_p)_{p }): \rho_{p_i}=b_i, \sigma_{p_i}=c_i \text{ for } 1\leq i\leq s\}.$$ Let $Q=\prod_{i=1}^s p_i$. The Chinese remainder theorem gives unique residues $b,c$ such that $m\equiv b_i(\text{mod}\  p_i), n\equiv c_i(\text{mod}\  p_i)$ for $1\leq i\leq s$ if and only if $$m\equiv b(\text{mod}\  Q), n\equiv c(\text{mod}\  Q).$$
   By \eqref{def of Psi}, 
   $$\delta_{\Psi_w(m,n)}(U)=\mathbf{1}_{m\equiv b(\text{mod}\  Q)}\cdot \mathbf{1}_{n\equiv c(\text{mod}\  Q)}.$$
   Therefore
$$\mu_{M,N}^w(U) 
    =\frac{1}{MN}\sum_{m=1}^M\sum_{n=1}^N\mathbf{1}_{m\equiv b(\text{mod}\  Q)}\cdot \mathbf{1}_{n\equiv c(\text{mod}\  Q)}. $$
    When $M,N\to \infty$, we have that $\mu_{M,N}^a(U)\to \frac{1}{Q^2}=\prod_{i=1}^sp_i^{-2}$, which complete the proof.

\end{proof}
Now we consider the Liouville function  $\lambda$.  We set $\lambda(n)=1$ if $n\leq 0$, then   $\lambda\in \{-1,1\}^{\Z}$. For $m,n>0$, by \eqref{continuous functions} and \eqref{def of Psi} $F(\Psi_\lambda(m,n))=\prod_{j=0}^k\lambda(m+a_jn)$, which implies that
\begin{equation}\label{intergral expression}
    \int Fd\mu^\lambda_{M,N}=D_k(M,N) \text{ and }\int G_pd\mu^\lambda_{M,N}=E_{k,p}(M,N).
\end{equation}

\subsection{Entropy, conditioning, and changes of labels}
The entropy argument below concerns finite random variables, even though the stationary space itself has infinitely many coordinates. We give the finite formulas explicitly so that every change of conditioning and every comparison of probability laws can be checked directly. In this section, all random variables take values in finite sets, and logarithms are natural. One can find the definitions and properties about entropy in e.g. \cite{D11}.

For a finite-valued random variable $U$, if \(U\) has probability mass function \(p_U(u)=\Prob(U=u)\), define the (Shannon) entropy of $U$ by 
\[
\H(U)=-\sum_{u:p_U(u)>0}p_U(u)\log p_U(u).
\]
Terms with zero mass are omitted, equivalently interpreted using \(0\log 0=0\). For a pair \((U,W)\), their joint probability mass function is
$p_{U,V}(u,v)=\Prob(U=u,V=v).$ Together with a value \(w\) of positive probability, write
\[
p_{U\mid W}(u\mid w)=\frac{p_{U,W}(u,w)}{p_W(w)}.
\]
The conditional entropy is
\[
\H(U\mid W)
=\sum_{w:p_W(w)>0}p_W(w)
\left(-\sum_{u:p_{U,W}(u,w)>0}p_{U\mid W}(u\mid w)\log p_{U\mid W}(u\mid w)\right).
\]
Both entropy and conditional entropy are nonnegative. For two probability measures $P,Q $ on a common finite set, write $P\ll   Q$ when $P(z)>0$ implies $Q(z)>0$ for any $z$ in the set. Their relative entropy is defined by
\begin{equation}\label{def of relative entropy}
    \D(P\|Q)=\sum_{z:P(z)>0}P(z)\log\frac{P(z)}{Q(z)},
\end{equation}
which is nonnegative. Indeed, as $-\log t\geq 1-t,$ by written $t=Q(z)/P(z)$ we have
$$\D(P\|Q)\geq \sum_{z:P(z)>0}(P(z)-Q(z))=1-\sum_{z:P(z)>0}Q(z)\geq 0.$$
Let $U,V$ be finite-valued random variables defined on a common probability space. One can define the entropy of $(U,V)$ by 
$$\H(U,V)=-\sum_{(u,v):p_{U,V}(u,v)>0}p_{U,V}(u,v)\log p_{U,V}(u,v).$$
We also have the property of entropy
\begin{equation}\label{entropy property}
    \H(U,V\mid W)=\H(U\mid W)+\H(V\mid U,W).
\end{equation}
The conditional mutual information is defined by 
\begin{equation}\label{def of information}
    \I(U;V\mid W)=\H(U\mid W)-\H(U\mid V,W).
\end{equation}
By \eqref{entropy property} we have 
\begin{equation}\label{def of information 2}
    \begin{aligned}
\I(U;V\mid W)
&=\H(U\mid W)+\H(V\mid W)-\H(U,V\mid W)\\
&=\H(V\mid W)-\H(V\mid U,W);
\end{aligned}
\end{equation}
In particular,
\begin{equation}\label{chain of information}
    \begin{aligned}
\I((U,V);W)
&=\H(U,V)-\H(U,V\mid W)\\
&=\H(U)+\H(V|U)-\H(U\mid W)-\H(V\mid U,W);\\
&=\big(\H(U)-\H(U\mid W)\big)+\big(\H(V|U)-\H(V\mid U,W)\big)\\
&=\I(U;W)+\I(V;W\mid U).
\end{aligned}
\end{equation}
It is also nonnegative. For each \(w\) of positive probability, the contribution of that conditional measure to the first line of \eqref{def of information 2} is
\begin{equation}\label{calculate I}
 \sum_{u,v:p_{U,V\mid W}(u,v\mid w)>0}
p_{U,V\mid W}(u,v\mid w)
\log\frac{p_{U,V\mid W}(u,v\mid w)}
{p_{U\mid W}(u\mid w)p_{V\mid W}(v\mid w)}.
\end{equation}

This follows by writing each marginal entropy as a sum against the joint conditional measures and then collecting the three logarithms. Averaging over \(w\) yields
\begin{equation}\label{def of information 3}
    \I(U;V\mid W)
=\sum_w p_W(w)\D\bigl(p_{U,V\mid W=w}\bigm\|p_{U\mid W=w}\otimes p_{V\mid W=w}\bigr).
\end{equation}

The sum is over \(w\) with \(p_W(w)>0\). Whenever a joint conditional mass at \((u,v)\) is positive, both of its marginal conditional masses are positive; thus the absolute-continuity hypothesis in every divergence in \eqref{def of information 3} holds.

Consequently, adding conditioning can only decrease entropy:
\[
\H(U\mid V,W)\le \H(U\mid W).
\]
Combining this fact with the chain rule gives conditional subadditivity,
\begin{equation}\label{ineq for entropy}
    \H(U,V\mid W)\le \H(U\mid W)+\H(V\mid W). 
\end{equation}
Neither assertion requires independence.

We will need two elementary invariance facts. First, if \(\beta\) is a bijection of the values of \(W\), then
\begin{equation}\label{inv of bijection for entropy}
  \H(U\mid \beta(W))=\H(U\mid W).   
\end{equation}
Indeed, conditioning on \(\beta(W)=\beta(w)\) is precisely conditioning on \(W=w\), with the same probability and the same conditional distribution of \(U\). A bijective relabeling of the values of \(U\) also leaves conditional entropy unchanged, since it merely reindexes the inner sum. This includes a fixed reordering of the coordinates of a random vector.

Second, if a transformation \(A\) preserves the underlying probability measure, then \((U\circ A,W\circ A)\) and \((U,W)\) have the same joint distribution: the probability of the former pair taking the value \((u,w)\) equals the measure of \(A^{-1}\{U=u,W=w\}\). Therefore
\begin{equation}\label{inv of bijection for entropy 2}
    \H(U\circ A\mid W\circ A)=\H(U\mid W). 
\end{equation}
These are distinct observations: \eqref{inv of bijection for entropy} is a relabeling of a conditioning variable, whereas \eqref{inv of bijection for entropy 2} uses invariance of the probability measure.

Finally, a random variable with values in a set of size \(d\) has entropy at most \(\log d\). To see this, compare its measure \(P\) with \(Q\) on that set:
\[
0\le \D(P\|Q)=\log d-\H(U).
\]
In particular, a $\{-1,1\}^{b}$-random variable has entropy at most \(b\log 2\), with no restriction on dependence between the signs. Its conditional entropy has the same upper bound, since conditioning decreases entropy.

\section{A summable estimate over prime intervals}
Proposition \ref{summable estimate} is proved by adapting the conditional entropy method of \cite{TT18}. Their work introduced finite conditional information estimates for correlations of multiplicative functions; we re-execute this framework in the two-parameter rectangular setting.

For $r\in\N_0$, put
\begin{equation}\label{def of H P R}
   H_r=2^r,\qquad
\cP_r=\{p\in\PP:H_r<p\le 2H_r\},\qquad
R_r=|\cP_r|.  
\end{equation}
For $z=((x_j)_{j=0}^{ k}, (\rho_p,\sigma_p)_{p\in\PP})\in X $, we let 
\begin{equation}\label{hatXr}
    \hat{X}_r(z)=(x_j(v))_{0\le j\le k,\;1\le v\le (a_k+1)H_r},
\end{equation}
\begin{equation}\label{VrYr}
    V_r(z)=((\rho_p,\sigma_p))_{p\le H_r},\qquad
Y_p(z)=(\rho_p,\sigma_p), \qquad Y_r(z)=\prod_{p\in \cP_r}Y_p
\end{equation}
\begin{equation}\label{Ir}
    I_r=\I(\hat{X}_r;Y_r\mid V_r).
\end{equation}

The sets \(\cP_r\) are nonempty by Bertrand--Chebyshev theorem (for \(r=0\), the prime is \(2\)). The prime number theorem will later give \(R_r/H_r\asymp 1/(r+1)\); the following estimate does not require that asymptotic.

\begin{prop}\label{summable estimate}
Let \(\mu\) be an \(S\)- and \(T\)-invariant probability measure on \(X\) such that $(\prod_{p}Y_p)_*\mu $ is the product uniform measure of $\prod_{p}(\Z/p\Z)^2$. For each $p\in\PP$, set
\[
e_p=\int G_p\,d\mu,\qquad
q_r=\frac{1}{R_r}\sum_{p\in\cP_r}|e_p|.
\]
Then
\[
\sum_{r=0}^{\infty}\frac{R_r}{H_r}q_r^2\le 50(k+1)(a_k+1)\log 2. 
\]
This proposition holds for every $k\in\N_0$, regardless of parity.
\end{prop}\subsection{The symbol blocks and their information budget}

All random variables in this section, unless a different measure is specified, are defined on \((X,\mu)\). Set

For definiteness the coordinates can be ordered lexicographically, with primes ordered increasingly. The vector \(V_0\) has no coordinates and is a constant random variable. With these conventions, \(V_{r+1}\) is the concatenation of \(V_r\) and \(Y_r\), up to a fixed reordering.

 The entropy bound gives
\begin{equation}\label{bounded for H}
   0\le \H(\hat{X}_r\mid V_r)\le \H(\hat{X}_r)\le (k+1)(a_k+1)H_r\log 2.  
\end{equation}

When \(u,v\in[H_r]\) and \(0\le j\le k\), the translated index \(u+a_jv\) lies in \([(a_k+1)H_r]\). Thus a two-parameter family of \(H_r^2\) products will depend on only \((k+1)(a_k+1)H_r\) signs. The number of signs grows linearly in \(H_r\). Now we have the information budget:

\begin{lem}\label{information budget}
With the preceding definitions,
\begin{equation}\label{bounded for sum}
\sum_{r=0}^{\infty}\frac{I_r}{H_r}\le (k+1)(a_k+1)\log 2.     
\end{equation}

Only \(S\)-invariance is needed for this lemma; neither \(T\)-invariance nor the product uniform measure of the residues is used.
\end{lem}
\begin{proof}
   Fix \(r\) and put \(d=(a_k+1)H_r\). For $z=((x_j)_{j=0}^{ k}, (\rho_p,\sigma_p)_{p }) $, the definition \eqref{def of S T} of action \(S\)  gives, for each symbolic coordinate  $x_j\in \{-1,1\}^{\Z}$,
\[
S^d z(x_j(v))=z(x_j(v+d)).
\]
Consequently \(\hat{X}_r\circ S^d\) consists of the signs with \(0\le j\le k\) and \(d+1\le v\le 2d\). The vector \(\hat{X}_{r+1}\) uses exactly the indices \(1\le v\le 2d\) for each \(j\). Splitting these indices into the two intervals shows that \(\hat{X}_{r+1}\) is a bijective reordering of
\[
(\hat{X}_r,\hat{X}_r\circ S^d).
\]

Let \(\beta_d\) act on values of \(V_{r+1}\) by
\[
\beta_d((\rho_p,\sigma_p)_{p\le 2H_r})
=(\rho_p+d,\sigma_p)_{p\le 2H_r}.
\]
Each addition takes place modulo its prime. The map is a bijection, with inverse \(\beta_{-d}\), and
\[
V_{r+1}\circ S^d=\beta_d(V_{r+1}).
\]
By \eqref{inv of bijection for entropy} and \eqref{inv of bijection for entropy 2}, one has that
\[
\H(\hat{X}_r\circ S^d\mid V_{r+1})
=\H(\hat{X}_r\circ S^d\mid V_{r+1}\circ S^d)
=\H(\hat{X}_r\mid V_{r+1}).
\]
Now we use the conditional subadditivity \eqref{ineq for entropy} to see that
\[
\begin{aligned}
\H(\hat{X}_{r+1}\mid V_{r+1})
&=\H(\hat{X}_r,\hat{X}_r\circ S^d\mid V_{r+1})\\
&\le \H(\hat{X}_r\mid V_{r+1})+\H(\hat{X}_r\circ S^d\mid V_{r+1})\\
&=2\H(\hat{X}_r\mid V_{r+1})\\
&=2\H(\hat{X}_r\mid V_r)-2I_r.
\end{aligned}
\]
The last line uses \(V_{r+1}=(V_r,Y_r)\) and the definition \eqref{def of information}. 

Normalize by the block length by setting
\[
b_r=\frac{\H(\hat{X}_r\mid V_r)}{H_r}.
\]
Since \(H_{r+1}=2^{r+1}=2H_r\), we have that
\[
b_{r+1}\le b_r-\frac{I_r}{H_r}.
\]
For any integer \(L\ge 0\), summing this inequality over \(0\le r\le L\), we obtain that
\[
\sum_{r=0}^{L}\frac{I_r}{H_r}\le b_0-b_{L+1}\le b_0=\frac{H(\hat{X}_0\mid V_0)}{H_0}\le (k+1)(a_k+1)\log 2.
\]
Here we use the fact that $b_{L+1}\geq 0$ and \eqref{bounded for H}. As $I_r,H_r\geq 0$, letting $L\to \infty$ we have \eqref{bounded for sum}.
\end{proof}

Now we can prove Proposition \ref{summable estimate}. We will use the notions \eqref{def of H P R},\eqref{hatXr}, \eqref{VrYr},\eqref{Ir} of $H_r,\cP_r,R_r,\hat{X}_r,V_r,Y_p,Y_r$ and $I_r$.
\begin{proof}[Proof of Proposition \ref{summable estimate}]
 For $z=((x_j)_{j=0}^{ k}, (\rho_p,\sigma_p)_{p\in \PP })\in X $, recall the definition of $G_p=(p^2\mathbf{1}_{\rho_p=\sigma_p=0}-1)\prod_{j=0}^kx_j(0)$, let $e_p=\int G_pd\mu$.

\medskip

 \textbf{Step 1: A single-prime.}
 
 Let $p\in \mathcal{P}_r$ be given.
 
 Define 
 $$W_p(z) =\frac{1}{H_r^2}\sum_{u=1}^{H_r}\sum_{v=1}^{H_r}
\left(\prod_{j=0}^{k}x_j(u+a_jv)\right)
\left(p^2\mathbf{1}_{\rho_p+u=\sigma_p+v=0(\text{mod}\  p)}-1\right).$$ As $\mu $ is $S,T$-invariant, one has that for any $u,v $,
$$\int \left(\prod_{j=0}^{k}x_j(u+a_jv)\right)
\left(p^2\mathbf{1}_{\rho_p+u=\sigma_p+v=0(\text{mod}\  p)}-1\right)d\mu=\int G_pd\mu=e_p.$$
Therefore $E_\mu(W_p)=e_p$. As $H_r<p\leq 2H_r$, for any $z\in X$, there exists at most one $(u,v) $ with $1\leq u,v\leq H_r$ such that $\rho_p+u=\sigma_p+v=0(\text{mod}\  p) $. Therefore $$|W_p(z)|\leq \frac{1}{H_r^2}\sum_{u=1}^{H_r}\sum_{v=1}^{H_r}|1|+\frac{1}{H_r^2}p^2\leq 5.$$
Let $U_p$ be the uniform measure on $(\Z/p\Z)^2 $. Let $P_p=(\hat{X}_r\times V_r\times Y_r)_*\mu$ and let $Q_p=(\hat{X}_r\times V_r)_*\mu\times U_p $, which are probability measures on the same finite set. 
By the definition of $P_p $, we have that $\E_{P_p}W_p=\E_{\mu}W_p=e_p $. 

We have that $ E_{Q_p}W_p=0$ since for any fixed $u,v $, precisely one of the $ p^2$ satisfies $(\rho_p,\sigma_p)\equiv (-u,-v)(\mod p) $. As $U_p$ is the uniform measure,
$$ E_{Q_p}W_p=p^2p^{-2}-1=0.$$ 
 If $P_p(x,w,y)>0$, then $(\hat{X}_r\times V_r)_*\mu (x,w)>0$ and therefore $Q_p(x,w,y)>0$. Thus we have $P_p\ll Q_p$. By Pinsker's inequality (see, for example, \cite[Example 4.9]{Han16}),
\[
\|P_p-Q_p\|_{TV}\le \sqrt{\frac12 \D(P_p\|Q_p)}.
\]
where $\|P_p-Q_p\|_{TV}=\frac{1}{2}\sum_{(x,w,y)}|P_p(x,w,y)-Q_p(x,w,y)|$.
Hence
\[
|e_p|^2
=|\E_{P_p}W_p|^2=
|\E_{P_p}W_p-\E_{Q_p}W_p|^2
\le 4\|W_p\|^2_\infty \|P_p-Q_p\|^2_{TV}
\le 50 \D(P_p\|Q_p).
\]

 We now prove that  \(\D(P_p\|Q_p)=\I(\hat{X}_r;Y_p\mid V_r)\). Denote $P_{X,V}=(\hat{X}_r\times V_r)_*\mu $, then $$P_p(x,w,y)=P(X,V)(x,w)P_p(y|(x,w))  \text{ and }Q_p=P(X,V)(x,w) U_p(y).$$ Thus by the definition of $D(P_q\| Q_p)$ (see \eqref{def of relative entropy}), we have that
\begin{equation*}
    \begin{array}{ll}{\displaystyle D(P_q\| Q_p)=\sum_{(x,w,y):P_p(x,w,y)>0}P_p(x,w,y)\log \frac{P_p(x,w,y)}{P_{X,V}(x,w)U_p(y)}}.
    \end{array}
\end{equation*} 
 By the assumption of this proposition, $(\prod_{p}Y_p)_*\mu $ is the product uniform measure. Then ${(Y_p)}_* \mu=U_p$. By \eqref{calculate I} we have $ D(P_q\| Q_p)=\I((\hat{X}_r,V_r);Y_p)$.

 As $ p>H_r$,  we have that $V_r $ and $Y_p$ are independent w.r.t. $\mu$. Then $ H(V_r\mid Y_p)=H(V_r)$, $\I(V_r;Y_p)=H(V_r)-H(V_r\mid Y_p)=0$ and therefore by \eqref{chain of information}
\[
\D(P_p\|Q_p)=\I((\hat{X}_r,V_r);Y_p)=\I(V_r;Y_p)+\I(\hat{X}_r;Y_p\mid V_r)=\I(\hat{X}_r;Y_p\mid V_r).
\] We have 
\begin{equation}\label{ep and DPQ}
    |e_p|^2\le 50 \D(P_p\|Q_p)= 50\I(\hat{X}_r;Y_p\mid V_r)
\end{equation}

\medskip

\noindent\textbf{Step 2: Summing over primes.}
Recall that  $(\prod_{p}Y_p)_*\mu $ is the product uniform measure. Thus for all \(p\in\mathcal P_r\), \(Y_p\) are mutually independent and independent of \(V\). Then we have
\[
\H(Y_r\mid V_r)=\sum_{p\in\mathcal P_r}\H(Y_p\mid V_r).
\]
On the other hand, by conditional subadditivity,
\[
\H(Y_r\mid \hat{X}_r,V_r)\le \sum_{p\in\mathcal P}\H(Y_p\mid \hat{X}_r,V_r).
\]
Subtracting the two inequalities gives
\[
\sum_{p\in\mathcal P_r}\I(\hat{X}_r;Y_p\mid V_r)
=
\sum_{p\in\mathcal P_r}\bigl(\H(Y_p\mid V)-\H(Y_p\mid X,V)\bigr)
\le \H(Y_r\mid V)-\H(Y_r\mid X,V)
=
\I(X;Y_r\mid V)=I_r.
\]
By \eqref{ep and DPQ} we have 
\[
\sum_{p\in\mathcal P_r}|e_p|^2
\le 50\sum_{p\in\mathcal P_r}\I(\hat{X}_r;Y_p\mid V_r)
\le 50 I_r.
\]
By Lemma \ref{information budget},
\[
\sum_{r=0}^\infty \frac{I_r}{H_r}\le (k+1)(a_k+1)\log 2.
\]
Then we have
\[
\sum_{r=0}^\infty \frac{1}{H_r}\sum_{p\in\mathcal P_r}|e_p|^2
\le 50\sum_{r=0}^\infty \frac{I_r}{H_r}
\le 50(k+1)(a_k+1)\log 2.
\]
Finally, by Cauchy--Schwarz inequality,
\[
q_r^2=\left(\frac{1}{R_r}\sum_{p\in\mathcal P_r}|e_p|\right)^2
\le \frac{1}{R_r}\sum_{p\in\mathcal P_r}|e_p|^2.
\]
Thus we have that 
\[
\sum_{r=0}^{\infty}\frac{R_r}{H_r}q_r^2\leq \sum_{r=0}^{\infty}\frac{1}{H_r}\sum_{p\in\mathcal P_r}|e_p|^2\le 50(k+1)(a_k+1)\log 2.
\]
 \end{proof}

\section{Constructing one continuous scale function}\label{Sec: scale function}

The entropy estimate is stated for a fixed invariant measure. The arithmetic identity compares different rectangles. We connect these two settings by extracting all fixed simultaneous dilations of a sequence of rectangles at once. The resulting function will be continuous in the logarithm of the dilation.

Recall that $$ D_k(M,N)=\E_{(m,n)\in [M]\times [N]}\prod_{j=0}^{k}\lambda(m+a_jn)$$  and
$$
    E_{k,p}(M,N)
=\E_{(m,n)\in [M]\times [N]}
(p^2\mathbf{1}_{p\mid m,p\mid n}-1)
\prod_{j=0}^{k}\lambda(m+a_jn).
$$
The factor \(p^2\) compensates for the density \(p^{-2}\) of the sublattice. Thus \(E_{k,p}\) measures the difference between the normalized restricted average and the original average.

The floor function, denoted by 
$\lfloor x \rfloor$, is defined for every real number 
$x$ as the greatest integer less than or equal to $x$. Suppose \(M,N\ge p\), and put \(M_p=\lfloor M/p\rfloor\), \(N_p=\lfloor N/p\rfloor\). Here and in the following, when we use $ \lfloor M/p\rfloor, \lfloor N/p\rfloor$, we always assume that $M,N>p $. Changing variables in the restricted sum gives
\[
\begin{aligned}
E_{k,p}(M,N)+D_k(M,N)
&=\frac{p^2}{MN}\sum_{1\leq m\le M,1\le n\le N}\mathbf{1}_{p\mid m,p\mid n}\prod_{j=0}^{k}\lambda(m+a_jn)\\
&=\frac{p^2}{MN}\sum_{u=1}^{M_p}\sum_{v=1}^{N_p}\prod_{j=0}^{k}\lambda(pu+a_jpv)\\
&=\frac{p^2}{MN}\sum_{u=1}^{M_p}\sum_{v=1}^{N_p}(\prod_{j=0}^{k}\lambda(p))\prod_{j=0}^{k}\lambda(u+a_jv)\\
&=-\frac{p^2}{MN}\sum_{u=1}^{M_p}\sum_{v=1}^{N_p}\prod_{j=0}^{k}\lambda(u+a_jv)
=-\frac{p^2M_pN_p}{MN}D_k(M_p,N_p).
\end{aligned}
\]
Consequently,
\begin{equation}\label{Ekp and Dk}
   E_{k,p}(M,N)=-c_p(M,N)D_k(M_p,N_p)-D_k(M,N),\qquad
c_p(M,N)=\frac{p^2M_pN_p}{MN}.  
\end{equation}

Write $M=pM_p+r_M$ and $N=pN_p+r_N$, with $0\leq r_M,r_N<p$. Then
\[
c_p(M,N)=\left(1-\frac{r_M}{M}\right)\left(1-\frac{r_N}{N}\right),
\]
so
\begin{equation}\label{bounded fof cp}
 0\le 1-c_p(M,N)
=\frac{r_M}{M}+\frac{r_N}{N}-\frac{r_Mr_N}{MN}
\le \frac{p-1}{M}+\frac{p-1}{N}.    
\end{equation}

In particular, \(c_p(M,N)\to 1\) when \(p\) is fixed and both side lengths tend to infinity.

We first give the elementary estimate that supplies this continuity.

\begin{lem}\label{Comparison of nested rectangles lemma}
Let \(b:\N^2\to[-1,1]\) and
\[
B(M,N)=\frac{1}{MN}\sum_{m\le M,n\le N}b(m,n).
\]
If \(1\le M'\le M\) and \(1\le N'\le N\), then
\begin{equation}\label{Comparison of nested rectangles}
 |B(M,N)-B(M',N')|\le 2\left(1-\frac{M'N'}{MN}\right).   
\end{equation}
\end{lem}

\begin{proof}
Let \(\alpha=M'N'/(MN)\). If \(\alpha=1\), the two rectangles coincide. Otherwise, let \(C\) be the average of \(b\) over the complement of \([M']\times[N']\) in \([M]\times[N]\). Decomposing the average into contributions from \([M']\times[N']\) and \(C\), we obtain
\[
B(M,N)=\alpha B(M',N')+(1-\alpha)C.
\]
Hence
\[
B(M,N)-B(M',N')=(1-\alpha)(C-B(M',N')).
\]
Both averages on the right lie in \([-1,1]\), so their difference has absolute value at most \(2\).
\end{proof}
\begin{lem}\label{Compactness of the profiles lemma}
For any sequence \(M_\ell,N_\ell\to\infty\), one can choose a subsequence, denoted in the same way, on which
\begin{equation}\label{Compactness of the profiles}
    f(t)=\lim_{\ell\to\infty}D_k(\lfloor e^{-t}M_\ell\rfloor,\lfloor e^{-t}N_\ell\rfloor)
\end{equation}
exists for every \(t\in\R\). The convergence is uniform on every compact interval, and
\begin{equation}\label{Compactness of the profiles bounded}
 |f(t)|\le 1,\qquad |f(t)-f(s)|\le 4|t-s|\qquad (s,t\in\R).    
\end{equation}
\end{lem}
\begin{proof}
Write
\[
M_\ell(t)=\lfloor e^{-t}M_\ell\rfloor,\qquad
N_\ell(t)=\lfloor e^{-t}N_\ell\rfloor.
\]
Define \(f_\ell(t)=D_k(M_\ell(t),N_\ell(t))\) when $M_\ell(t),N_\ell(t)\geq 1$, and set \(f_\ell(t)=0\) otherwise. On any fixed \(t\in [-T,T]\) we have $M_\ell(t),N_\ell(t)\geq 1$ for all sufficiently large \(\ell\).

We begin with an explicit estimate including the floor errors. Fix \(-T\le s\le t\le T\), put \(\alpha=e^{-(t-s)}\), and suppose \(M_\ell,N_\ell\ge 2e^T\). Using the fractional part notation \(\{x\}=x-\lfloor x\rfloor\) and the fact that $\alpha \lfloor x\rfloor=\alpha x-\alpha\{x\}=\lfloor\alpha  x\rfloor+\{\alpha x\}-\alpha\{x\}$, we have
\[
M_\ell(t)-\alpha M_\ell(s)=\lfloor e^{-t}M_\ell\rfloor-e^{-(t-s)}\lfloor e^{-s}M_\ell\rfloor=\alpha\{e^{-s}M_\ell\}-\{e^{-t}M_\ell\}.
\]
Thus $|M_\ell(t)-\alpha M_\ell(s)|\leq 1$. Also,
\[
M_\ell(s)\ge e^{-T}M_\ell-1\ge \frac12 e^{-T}M_\ell.
\]
Consequently,
\[
\left|\frac{M_\ell(t)}{M_\ell(s)}-\alpha\right|=\left|\frac{M_\ell(t)-\alpha M_\ell(s)}{M_\ell(s)}\right|
\le \frac{1}{M_\ell(s)}
\le \frac{2e^T}{M_\ell}.
\]
Similarly,
\[
\left|\frac{N_\ell(t)}{N_\ell(s)}-\alpha\right|
\le \frac{2e^T}{N_\ell}.
\]
Let  \(u_\ell=\frac{M_\ell(t)}{M_\ell(s)},v_\ell=\frac{N_\ell(t)}{N_\ell(s)}\). Since \(s\le t\), they lie in \([0,1]\).  Therefore
\[
|u_\ell v_\ell-\alpha^2|
=|(u_\ell-\alpha)v_\ell+\alpha(v_\ell-\alpha)|
\le 2e^T(M_\ell^{-1}+N_\ell^{-1}).
\]
As $M_\ell(t)\leq M_\ell(s)$, $N_\ell(t)\leq N_\ell(s)$, we applying Lemma \ref{Comparison of nested rectangles lemma} to the array \(b(m,n)=\prod_{j=0}^{k}\lambda(m+a_jn)\), we obtain
\[
\begin{aligned}
|f_\ell(t)-f_\ell(s)|
&\le 2(1-u_\ell v_\ell)\\
&\le 2(1-\alpha^2)+4e^T(M_\ell^{-1}+N_\ell^{-1})\\
&\le 4(t-s)+4e^T(M_\ell^{-1}+N_\ell^{-1}),
\end{aligned}
\]
where \(1-e^{-2v}\le 2v\) was used for \(v=t-s \ge 0\). Interchanging \(s,t\) proves
\begin{equation}\label{bound for ft-fs}
  |f_\ell(t)-f_\ell(s)|
\le 4|t-s|+\epsilon_\ell(T),\qquad
\epsilon_\ell(T)=4e^T(M_\ell^{-1}+N_\ell^{-1}). 
\end{equation}

For each fixed \(T\), \(\epsilon_\ell(T)\to 0\). In particular, this estimate remains valid when \(M_\ell/N_\ell\) tends to zero or infinity.

We now choose the subsequence. Enumerate \(\mathbb{Q}=\{\xi_1,\xi_2,\dots\}\), with \(\xi_1=0\). Since \(f_\ell(\xi_1)\) belongs to \([-1,1]\)  for every $\ell$, we can choose $J_1\subset \N$ such that ${\displaystyle \lim_{\ell\in J_1,\ell \to \infty}f_\ell(\xi_i)} $ exists. Inductively, we choose nested infinite sets of indices
\[
J_1\supset J_2\supset\cdots
\]
such that \(f_\ell(\xi_i)\) converges as \(\ell\) tends to infinity through \(J_i\). Choose increasing indices \(\ell_i\in J_i\). For every fixed \(j\), all \(\ell_i\) with \(i\ge j\) belong to \(J_j\). Thus the subsequence \(\{f_{\ell_i}\}\) converges at every rational point. We relabel this subsequence as \((f_\ell)\) and denote its rational limit by \(g(\xi)\).

For rational \(s,t\), choose \(T\) containing them and pass to the limit in \eqref{bound for ft-fs}. This gives
\[
|g(t)-g(s)|\le 4|t-s|.
\]
For any real \(t\), choose rational numbers \(\zeta_i\to t\). The sequence \(g(\zeta_i)\) is Cauchy by this inequality, so it has a real limit. If \(\zeta_i'\to t\) is another rational sequence, then \(|g(\zeta_i)-g(\zeta_i')|\le 4|\zeta_i-\zeta_i'|\to 0\). The limit is therefore independent of the approximation. Define it to be \(f(t)\). Passing to limits in the rational inequality proves \eqref{Compactness of the profiles bounded}.

It remains to prove convergence of \(\{f_\ell\}\) at all real points, and uniformly on compact intervals. Fix \(T>0\) and \(0<\eta<1\). Choose a finite set $S$ of rational points in \([-T-1,T+1]\) which is \(\eta/16\)-dense in \([-T,T]\).  For large enough \(\ell\), both
\[
\max_{s\in S}|f_\ell(s)-f(s)|<\eta/4,\qquad
\epsilon_\ell(T+1)<\eta/4
\]
hold. For any $t\in [-T,T] $, let $ s\in S$ be such that $|t-s|<\eta/16$. Then
\[
\begin{aligned}
|f_\ell(t)-f(t)|
&\le |f_\ell(t)-f_\ell(s)|+|f_\ell(s)-f(s)|+|f(s)-f(t)|\\
&\le 4|t-s|+\epsilon_\ell(T+1)+\eta/4+4|t-s|\\
&<8(\eta/16)+\eta/4+\eta/4=\eta.
\end{aligned}
\]
Notice that the choice of \(\ell\) is uniform in \(t\in[-T,T]\), proving the asserted locally uniform convergence.
\end{proof}
There are two distinct compactness constructions. Lemma \ref{Compactness of the profiles lemma} selects a single function \(f\) at all real scales; it does not select measures at all real scales. We now fix a scale, identify its discrepancies, and only then choose a stationary measure.
\begin{lem}\label{Arithmetic discrepancies of the profile lemma}
For the subsequence in Lemma \ref{Compactness of the profiles lemma}, every fixed \(t\in\R\) and fixed prime \(p\) satisfy
\begin{equation}\label{Arithmetic discrepancies of the profile}
  \lim_{\ell\to\infty}E_{k,p}(M_\ell(t),N_\ell(t))
=-f(t+\log p)-f(t).  
\end{equation}

\end{lem}

\begin{proof}
For \(x\ge 0\) and any positive integer \(p\), we have
\begin{equation}\label{floor equation}
    \left\lfloor\frac{\lfloor x\rfloor}{p}\right\rfloor
=\left\lfloor\frac{x}{p}\right\rfloor. 
\end{equation}
Indeed, write \(\lfloor x\rfloor=pa+b\) with \(0\le b<p\), and write \(x=pa+b+u\), \(0\le u<1\). Since \(0\le b+u<p\), both floors equal \(a\). Therefore
\[
\lfloor M_\ell(t)/p\rfloor=M_\ell(t+\log p),\qquad
\lfloor N_\ell(t)/p\rfloor=N_\ell(t+\log p).
\]
There is no additional floor error in these identities.

For large enough \(\ell\), \(M_\ell(t),N_\ell(t)\ge p\). Apply \eqref{Ekp and Dk} for  $M_\ell(t),N_\ell(t)$, we have
$$ E_{k,p}(M_\ell(t),N_\ell(t))=-c_{l,p}(t)D_k(\lfloor \frac{M_\ell(t)}{p} \rfloor,\lfloor \frac{N_\ell(t)}{p}\rfloor)-D_k(M_\ell(t),N_\ell(t)),$$
where $${\displaystyle c_{l,p}(t)=\frac{p^2\lfloor \frac{M_\ell(t)}{p} \rfloor\lfloor \frac{N_\ell(t)}{p}\rfloor}{M_\ell(t)N_\ell(t)}. }$$  
Recall that $M_\ell(t)=\lfloor e^{-t}M_\ell\rfloor,
N_\ell(t)=\lfloor e^{-t}N_\ell\rfloor,$ then $\lfloor \frac{M_\ell(t)}{p} \rfloor=M_\ell(t+\log p) $ and $\lfloor \frac{N_\ell(t)}{p} \rfloor=N_\ell(t+\log p) $. Thus we have
\[
c_{\ell,p}(t)=\frac{p^2M_\ell(t+\log p)N_\ell(t+\log p)}
{M_\ell(t)N_\ell(t)}
\]
and therefore
\[
E_{k,p}(M_\ell(t),N_\ell(t))
=-c_{\ell,p}(t)f_\ell(t+\log p)-f_\ell(t).
\]
By \eqref{bounded fof cp},
\[
0\le 1-c_{\ell,p}(t)
\le \frac{p-1}{M_\ell(t)}+\frac{p-1}{N_\ell(t)}.
\]
Hence
\[
\begin{aligned}
&|E_{k,p}(M_\ell(t),N_\ell(t))+f(t+\log p)+f(t)|\\
&\le (\frac{p-1}{M_\ell(t)}+\frac{p-1}{N_\ell(t)})| f_\ell(t+\log p)|
+|f_\ell(t+\log p)-f(t+\log p)|
+|f_\ell(t)-f(t)|.
\end{aligned}
\]
As $|f_\ell|\leq 1 $ and $f_\ell(t)\to f(t)$ for fixed $t$, all four terms tend to zero. This proves the lemma. The prime \(p\) is fixed throughout this limit.
\end{proof}
\begin{prop}\label{The entropy estimate at every profile scale proposition}
For every  \(f\) from Lemma \ref{Compactness of the profiles lemma} and every \(t\in \R\),
\begin{equation}\label{The entropy estimate at every profile scale}
    \sum_{r=0}^{\infty}\frac{R_r}{H_r}A_r(t)^2\le 50(k+1)(a_k+1)\log 2,\text{ where }
A_r(t)=\frac{1}{R_r}\sum_{p\in\cP_r}|f(t+\log p)+f(t)|. 
\end{equation}

\end{prop}

\begin{proof}
Fix \(t\). The side lengths \(M_\ell(t),N_\ell(t)\) tend to infinity, so Lemma \ref{invariant measure property} supplies a further subsequence on which
\[
\mu_{M_\ell(t),N_\ell(t)}^\lambda\to \mu_t.
\]
The limit is invariant under \(S,T\) and has product uniform residue marginal.

For each fixed prime \(p\), weak* convergence and \eqref{intergral expression} give
\[
\int G_p\,d\mu_t
=\lim_{\ell\to\infty}\int G_p\,d\mu_{M_\ell(t),N_\ell(t)}
=\lim_{\ell\to\infty}E_{k,p}(M_\ell(t),N_\ell(t)).
\]
The last limit is already determined by Lemma \ref{Arithmetic discrepancies of the profile lemma}, even before passing to this further subsequence. Thus
\[
e_p(t):=\int G_p\,d\mu_t=-f(t+\log p)-f(t)
\]
for every prime \(p\). All these identities concern the same measure \(\mu_t\). Although \(\|G_p\|_\infty\) grows with \(p\), weak* convergence is being applied to one fixed continuous function at a time.

Now apply Proposition \ref{summable estimate} directly to the already constructed measure \(\mu_t\). For each finite \(L\) it gives
\[
\sum_{r=0}^{L}\frac{R_r}{H_r}
\left(\frac{1}{R_r}\sum_{p\in\cP_r}|e_p(t)|\right)^2
\le 50(k+1)(a_k+1)\log 2.
\]
Substitute the identified values of \(e_p(t)\) and let \(L\to\infty\). The summands are nonnegative, so the finite bound passes to the increasing limit and proves \eqref{The entropy estimate at every profile scale}.

There is no interchange of the rectangle limit with an infinite prime sum. The sum is estimated on a fixed limiting measure after that measure has been constructed. The further subsequence and \(\mu_t\) may depend on \(t\); the function \(f\) does not, since all its values were fixed by Lemma \ref{Compactness of the profiles lemma}. We never need a simultaneous or measurable choice of the measures \(\mu_t\).
\end{proof}

\section{Overlapping prime intervals force the profile to vanish}

At this stage the arithmetic variables have disappeared. By \eqref{Compactness of the profiles}, we have a \(4\)-Lipschitz function \(f:\R\to[-1,1]\) satisfying \eqref{The entropy estimate at every profile scale}. We prove an analytic assertion about this function. The argument has three ingredients: replacing prime averages by integrals uniformly for the moving test functions, selecting a common sequence of good prime intervals for two scales, and comparing the overlapping intervals on the logarithmic axis.
\begin{lem}\label{Uniform comparison of prime averages lemma}
For any \(4\)-Lipschitz \(f:\R\to[-1,1]\), let
\[
A_r(t)=\frac{1}{R_r}\sum_{H_r<p\le 2H_r}|f(t+\log p)+f(t)|
\]
and
\begin{equation}\label{Uniform comparison of prime averages}
    J_r(t)=\int_1^2|f(t+\log H_r+\log u)+f(t)|du. 
\end{equation}
There are numbers \(\epsilon_r\to 0\), independent of \(f,t\), such that
\begin{equation}\label{error between Ar and Jr}
    |A_r(t)-J_r(t)|\le \epsilon_r. 
\end{equation}

\end{lem}
\begin{proof}
    Normalize the primes in the interval to points in \([1,2]\): Since $p/H_r\in [1,2]$, we define the measure supported on $[1,2]$ by
\[
\nu_r=\frac{1}{R_r}\sum_{H_r<p\le 2H_r}\delta_{p/H_r}.
\]
Let $\pi:\R\to \N $ be the prime-counting function that $\pi(x)=\#\{p: p\leq x \text{ and }p \text{ prime}\}$. Recall the definition \eqref{def of H P R} of $R_r$, we have $R_r=\pi(2H_r)-\pi(H_r)$. For any fixed $c$, by the prime number theorem we have
\[
\pi(cH_r)=\frac{cH_r}{\log cH_r}(1+o(1))=\frac{cH_r}{\log H_r}(1+o(1)).
\]
Thus for any $a<b\in [1,2]$,  
$$\nu_r((a,b])= \frac{1}{R_r}(\pi(bH_r)-\pi(aH_r))=\frac{\pi(bH_r)-\pi(aH_r)}{\pi(2H_r)-\pi(H_r)}
\to b-a.$$ 
The limiting measure is Lebesgue measure on \([1,2]\), of total mass one.

We need uniformity of the limit because our test function depends on \(r\). Let \(\cG\) be the set of all $4$-Lipschitz functions $g:[1,2]\to[-2,2]$. We want to show that 
\begin{equation}\label{uniformly bound}
    \sup_{g\in \cG}\left|\int_1^2 g(u)d\nu_r(u)-\int_1^2g(u)du\right|\to 0
\end{equation}
Fix an integer \(L\ge 1\). Set
\[
u_i=1+\frac{i}{L}\quad (0\le i\le L),\qquad
I_i=(u_{i-1},u_i]\quad (1\le i\le L).
\]
Both \(\nu_r\) and Lebesgue measure have zero mass at \(1\), so these intervals partition their support up to a null set. For \(g\in\cG\), replace \(g\) on \(I_i\) by \(g(u_{i-1})\), obtaining a step function \(g_L\). As $g$ is $4$-Lipschitz, one has that \(|g-g_L|\le 4/L\) on every \(I_i\). Thus
\[
\begin{aligned}
\left|\int_1^2 g(u)\,d\nu_r(u)-\int_1^2 g(u)\,du\right|
&\le \frac{8}{L}
+\left|\sum_{i=1}^{L}g(u_{i-1})
\left(\nu_r(I_i)-\frac{1}{L}\right)\right|\\
&\le \frac{8}{L}+2\sum_{i=1}^{L}
\left|\nu_r(I_i)-\frac{1}{L}\right|.
\end{aligned}
\]
In particular,
\begin{equation}\label{uniformly bound L}
 \sup_{g\in\cG}\left|\int_1^2 g(u)\,d\nu_r(u)-\int_1^2 g(u)\,du\right|
\le \frac{8}{L}+2\sum_{i=1}^{L}
\left|\nu_r(I_i)-\frac{1}{L}\right|.   
\end{equation}

Given \(\eta>0\), first choose \(L\) so that \(8/L<\eta/2\). For this fixed partition, the prime number theorem makes the remaining sum smaller than \(\eta/2\) once \(r\) is large. The supremum in \eqref{uniformly bound L} therefore tends to zero.

Finally, define
\[
g_{r,t}(u)=|f(t+\log H_r+\log u)+f(t)|.
\]
It satisfies \(|g_{r,t}|\le 2\), and the reverse triangle inequality gives
\[
|g_{r,t}(u)-g_{r,t}(v)|
\le 4|\log u-\log v|
\le 4|u-v|.
\]
So \(g_{r,t}\in\cG\) for every \(r,t,f\) under consideration. Its integral against \(\nu_r\) is exactly \(A_r(t)\), and its Lebesgue integral is \(J_r(t)\). Taking \(\epsilon_r\) to be the supremum in \eqref{uniformly bound L} proves the statement.
\end{proof}
\begin{lem}\label{common sequence of small terms}
Let \(w_r>0\) satisfy \(\sum_r w_r=\infty\), and let \(a_r,b_r\ge 0\) satisfy
\[
C:=\sum_r w_r(a_r^2+b_r^2)<\infty.
\]
There exist increasing indices \(r_\ell\to\infty\) such that
\[
a_{r_\ell}+b_{r_\ell}<1/\ell.
\]
\end{lem}

\begin{proof}
For each positive integer \(\ell\), define the exceptional set
\[
E_\ell=\{r:a_r^2+b_r^2\ge 1/(2\ell^2)\}.
\]
We have \[
\frac{1}{(2\ell^2)}\sum_{r\in E_\ell}w_r
\le \sum_{r\in E_\ell}w_r(a_r^2+b_r^2)\le C
\]
 and therefore $\sum_{r\notin E_{\ell}}w_r=\infty $. Thus $\N\setminus E_\ell  $ are infinite sets for each $ \ell$. Then we can choose $ r_\ell$ inductively by $r_{\ell}\notin E_\ell$ and $r_{\ell}>r_{\ell-1}$. Then
\[
a_{r_\ell}+b_{r_\ell}
\le \sqrt{2(a_{r_\ell}^2+b_{r_\ell}^2)}<1/\ell.
\]
\end{proof}

\begin{lem}\label{The overlap inequality lemma}
For $f$ and  \(J_r\) from Lemma \ref{Uniform comparison of prime averages lemma} and \(0<h<\log 2\),
\begin{equation}\label{The overlap inequality}
    (\log 2-h)|f(t+h)-f(t)|\le J_r(t)+J_r(t+h)
\qquad (t\in\R,r\ge 0).
\end{equation}
\end{lem}

\begin{proof}
The logarithms of the points in \([H_r,2H_r]\), translated by \(t\), form the interval
\[
I_r(t)=[t+\log H_r,t+\log H_r+\log 2].
\]
In \eqref{Uniform comparison of prime averages} make the change of variables \(w=t+\log H_r+\log u\). Then
\[
J_r(t)=\int_{I_r(t)}|f(w)+f(t)|e^{w-t-\log H_r}\,dw.
\]
Note that $e^{w-t-\log H_r}\in[1,2]$ when $w\in I_r(t)$. Therefore
\begin{equation}\label{w t and J}
\int_{I_r(t)}|f(w)+f(t)|\,dw\le J_r(t).     
\end{equation}

Apply the same identity at \(t+h\). The two intervals intersect in
\[
K_r(t,h)=I_r(t)\cap I_r(t+h)
=[t+h+\log H_r,t+\log H_r+\log 2],
\]
which has length \(\log 2-h>0\). For every $w\in K_r(t,h)$,
\[
|f(t+h)-f(t)|
=|(f(w)+f(t+h))-(f(w)+f(t))|
\le |f(w)+f(t+h)|+|f(w)+f(t)|.
\]
Integrate over \(K_r(t,h)\) for $w$. 
Using \eqref{w t and J} at both scales proves
\[
\begin{aligned}
(\log 2-h)|f(t+h)-f(t)|
&\le \int_{I_r(t+h)}|f(w)+f(t+h)|\,dw\\
&\quad+\int_{I_r(t)}|f(w)+f(t)|\,dw\\
&\le J_r(t+h)+J_r(t).
\end{aligned}
\]
\end{proof}
\begin{prop}\label{Vanishing of a finite-energy profile prop}
Suppose \(f:\R\to[-1,1]\) is \(4\)-Lipschitz and, for every fixed \(t\in\R\),
\begin{equation}\label{Vanishing of a finite-energy profile}
 \sum_{r=0}^{\infty}\frac{R_r}{H_r}A_r(t)^2<\infty,\qquad
A_r(t)=\E_{p\in\cP_r}|f(t+\log p)+f(t)|.   
\end{equation}
Then \(f\equiv 0\).
\end{prop}

\begin{proof}
Let $\pi$ be the prime-counting function. The prime number theorem gives
\[
R_r=\pi(2H_r)-\pi(H_r)=\frac{H_r}{\log H_r}(1+o(1)).
\]
Recall that $H_r=2^r$ and therefore \(\log H_r=r\log 2\), this implies
\begin{equation}\label{condition R/H}
   \frac{R_r}{H_r}\sim \frac{1}{r\log 2}\quad (r\to\infty),\qquad
\sum_{r=0}^{\infty}\frac{R_r}{H_r}=\infty. 
\end{equation}
Only the asymptotic as \(r\to\infty\) is meant in the first formula; the finite initial indices cause no issue.

Fix \(t\in\R\) and \(0<h<\log 2\). Applying the condition \eqref{Vanishing of a finite-energy profile} at \(t\) and at \(t+h\) gives
\[
\sum_r \frac{R_r}{H_r}(A_r(t)^2+A_r(t+h)^2)<\infty.
\]
Use Lemma  \ref{common sequence of small terms} with \(w_r=R_r/H_r\), \(a_r=A_r(t)\), and \(b_r=A_r(t+h)\). There are increasing \(r_\ell\) such that
\begin{equation}\label{At+h +At }
 A_{r_\ell}(t)+A_{r_\ell}(t+h)<1/\ell.    
\end{equation}

By Lemma \ref{Uniform comparison of prime averages lemma},
\begin{equation}\label{Jt+Jt+h}
    0\le J_{r_\ell}(t)+J_{r_\ell}(t+h)
\le A_{r_\ell}(t)+A_{r_\ell}(t+h)+2\epsilon_{r_\ell}
\xrightarrow{\ell \to \infty} 0.
\end{equation}
We use Lemma \ref{The overlap inequality lemma} to find that 
\[
0\le |f(t+h)-f(t)|
\le \frac{J_{r_\ell}(t)+J_{r_\ell}(t+h)}{\log 2-h}
 \xrightarrow{\ell \to \infty} 0.
\]
The left side does not depend on \(\ell\), so it must be zero. Thus we have proved $f$ is $h$-periodic for every $h\in(0,\log 2)$, which implies $f\equiv C$ for some constant $C$.


For a constant function,
\[
A_r(t)=\frac{1}{R_r}\sum_{p\in\cP_r}|C+C|=2|C|.
\]
Substituting into condition \eqref{Vanishing of a finite-energy profile} gives
\[
4|C|^2\sum_r \frac{R_r}{H_r}<\infty.
\]
By \eqref{condition R/H} $\sum_r \frac{R_r}{H_r}=\infty$. Thus \(C=0\), as required.
\end{proof}
\section{Completion}

\begin{proof}[Proof of Theorem \ref{main theorem}]
Fix even \(k\ge 2\). The failure of \eqref{main} means that there is \(\delta>0\) such that for every integer \(L\) one can find \(M,N\ge L\) with \(|D_k(M,N)|\ge\delta\). Choose such a pair \((M_\ell,N_\ell)\) with \(M_\ell,N_\ell\ge\ell\). Then both side lengths tend to infinity, while
\[
|D_k(M_\ell,N_\ell)|\ge\delta\qquad (\ell\ge 1).
\]
Lemma \ref{Compactness of the profiles lemma} gives a subsequence and a common \(4\)-Lipschitz profile \(f\). At \(t=0\), the floored side lengths equal the original integers. Hence
\[
|f(0)|=\left|\lim_{\ell\to\infty}D_k(M_\ell,N_\ell)\right|
=\lim_{\ell\to\infty}|D_k(M_\ell,N_\ell)|\ge\delta.
\]
But Proposition \ref{Vanishing of a finite-energy profile prop} gives \(f\equiv 0\), a contradiction. Therefore \eqref{main} holds.
\end{proof}

\section*{Acknowledgments}
Q.H. was partially supported by National Key Research and Development Program of China 2024YFA1013600; China Postdoctoral Science Foundation 2025M773065; Funda-mental Research Funds for the Central Universities WK0010250102. 
Z.L. was partially supported by National Natural Science Foundation of China 12471187.
K.O. was partially supported by National Natural Science Foundation of China (12371196, 12426201).

\textit{Statement on the Use of AI Tools in This Work}: The authors used AI tools during the preparation of this manuscript. In particular, AI assistance played a significant role in the development of Proposition \ref{Vanishing of a finite-energy profile prop}, where it helped explore and refine the analytic argument showing that a bounded Lipschitz function satisfying the bounded condition must vanish.  All mathematical results, proofs, and interpretations were verified and approved by the authors, who take full responsibility for the content of this work.
\bibliographystyle{alpha}
\bibliography{ref.bib,universal_bib}
\end{document}